\documentclass[11pt]{amsart}
\usepackage[margin= 1.3 in]{geometry}

\usepackage{pdfsync}

 \usepackage[plainpages=false]{hyperref}
 \usepackage{amsfonts,amsmath,amssymb,amsthm,accents,mathtools}
 \usepackage{latexsym,lscape,rawfonts,mathrsfs}

 \usepackage[all]{xy}
 \usepackage{eufrak}
 \usepackage{graphicx,psfrag}
 \usepackage{pstool}

 \usepackage{array,tabularx}

 \usepackage{appendix}

\usepackage{txfonts}

 \newcommand{\ba}{\begin{align}}
 \newcommand{\ea}{\end{align}}
 \newcommand{\bal}{\begin{align*}}
 \newcommand{\eal}{\end{align*}}

 \DeclareMathOperator{\diam}{diam}

 \newcommand{\Rm}{\mathbf{Rm}}

 \newcommand{\dvol}{\text{d}V}

\renewcommand{\epsilon}{\varepsilon}

 \makeatletter
 \def\ExtendSymbol#1#2#3#4#5{\ext@arrow 0099{\arrowfill@#1#2#3}{#4}{#5}}
 
 \makeatother

 \makeatletter
 \def\ExtendSymbol#1#2#3#4#5{\ext@arrow 0099{\arrowfill@#1#2#3}{#4}{#5}}
 
 \makeatother

 \definecolor{hao}{rgb}{1,0.5,0}
 \definecolor{miao}{cmyk}{0.5,0,0.2,0.2}
 \definecolor{qiao}{gray}{0.96}

\newtheorem{prop}{Proposition}[section]

\newtheorem{theorem}[prop]{Theorem}

\newtheorem{lemma}[prop]{Lemma}
\newtheorem{claim}[prop]{Claim}

\newtheorem*{theorem*}{Theorem}
\theoremstyle{remark}
\newtheorem{remark}{Remark}

\numberwithin{equation}{section}


\address{Shaosai Huang, Department of Mathematics, University of
Wisconsin-Madison, 480 Lincoln Drive, Madison, WI, 53706, U.S.A.}
\email{sshuang@math.wisc.edu}

\address{Selin Ta\c{s}kent, Department of Mathematics, Stony Brook University,
100 Nicolls Road, Stony Brook, NY, 11794, U.S.A.}
\email{selin@math.stonybrook.edu}

\title[Eigenfunctions on collapsing Einstein manifolds]{Small fiberwise
oscillation of the eigenfunctions of collapsing Einstein manifolds}
\author{Shaosai Huang}
\author{Selin Ta\c{s}kent}
\date{\today}

\begin{document}
\maketitle

\begin{abstract}
By Cheeger-Colding's almost splitting theorem, if a domain in a Ricci flat
manifold is pointed-Gromov-Hausdorff close to a lower dimensional Euclidean
domain, then there is a harmonic almost splitting map. We show that any
eigenfunction of the Laplace operator is almost constant along the fibers of
the almost splitting map, in the $L^2$-average sense. This generalizes an
estimate of Fukaya in the case of collapsing with bounded diameter and
sectional curvature.
\end{abstract}

\section{Introduction and backgroundd}
A natural and important theme in geometric analysis is to understand the
uniform behavior of the eigenvalues and eigenfunctions of the Laplace operators 
associated to a given family of Riemannian manifolds. In the seminal work
\cite{Fukaya87b}, by introducing the concept of measured Gromov-Hausdorff
convergence, Fukaya proved that the eigenvalues of the Laplace operators of
Riemannian manifolds are continuous with respect to such topology, provided
that the manifolds in consideration have uniformly bounded diameter and
sectional curvature. In proving such continuity, the main difficulty is that a
sequence of Riemannian manifolds may collapse in volume, and a key tool to
overcome this difficulty is the following functional inequality, which is
referred to as the Key Lemma in Fukaya's work \cite[\S 3]{Fukaya87b}:
\begin{theorem}[Fukaya's Key Lemma]\label{thm: Fukaya} 
Suppose a sequence of $m$-dimensional Riemannian manifolds $\{(M_i,g_i)\}$
satisfy the regularity assumptions
\begin{align}
\forall l\in \mathbb{N},\quad \exists C_l>0,\quad
\sup_{M_i}\|\nabla^l\Rm_{g_i}\|\ \le\ C_l, \label{eqn: regularity}
\end{align}
and suppose that there is a $k$-dimensional $(k<m)$ closed Riemannian
manifold $(N,h)$, such that
\begin{align*}
\lim_{i\to\infty} d_{GH}(M_i,N)\ =\ 0.
\end{align*}
Then for any $i$ sufficiently large, and any $u\in C^{\infty}(M_i)$, we have
the estimate
\begin{align}
\|\nabla^Tu\|_{L^{\infty}(M_i)}\ \le\
C_F(m,k)\left(\|u\|_{\bar{L}^2(M_i)}+\|\Delta_{g_i}^qu\|_{\bar{L}^2(M_i)}\right)
d_{GH}(M_i,N),
\end{align}
for some dimensional constants $C_F(m,k)>0$ and some large $q\in \mathbb{N}$,
with $\|\cdot \|_{\bar{L}^2(M_i)}$ denoting the $L^2$-average of a given
function on $M_i$.
\end{theorem}

Here we recall that when the Riemannian manifold $(M_i,g)$ is sufficiently close
to, in the Gromov-Hausdorff sense, another lower diemnsional Riemannian
manifold $(N,h)$, then the regularity assumptions (\ref{eqn: regularity})
guarantee that there is a fibration $\Phi_i:M_i\to N$, which is also an almost
Riemannian submersion. Moreover, the $\Phi_i$ fibers, as embedded submanifolds
in $M_i$, are all homeomorphic to some infranil manifold. Here
for any $x\in M_i$, $\nabla^Tu(x)\in T_x\Phi^{-1}_i(\Phi_i(x))$ denotes the
restriction of $\nabla u(x)$ to the directions tangential to the fiber of
$\Phi_i$ through $x$.

Roughly speaking, Theorem~\ref{thm: Fukaya} tells that when a sequence of
Riemannian manifolds collapses to a lower dimensional one with (\ref{eqn:
regularity}) satisfied, then on those sufficiently collapsed ones in the
sequence, any ``\emph{reasonable}'' function behaves almost like constants
along the fiber directions.  In fact, besides playing a crucial rule in proving
the continuity of the eigenvalues, the Key Lemma provides more information than
needed: when the collapsing limit is a manifold, the eigenfunctions of the
collapsing sequence converge in the $C^1$-sense (see \cite[\S 3]{Fukaya89}),
and this in turn helped prove that the collapsing limit can be embedded into a
finite dimensional Euclidean space by heat kernel methods (see \cite[\S
4]{Fukaya89}).

The curvature assumption (\ref{eqn: regularity}) is crucial here --- withouth
this condition, we cannot expect any topological structure of $M_i$ relating to
$N$, and Fukaya's proof heavily relies on the fact that the fibers are infranil
manifolds: he worked locally on a tangent space of a point, on which the
pull-back metrics do not collapse, and the pull-back functions under
consideration are locally periodic with shorter and shorter periods (see
\cite[\S 3]{Fukaya87b}).

However, with many natural examples of Einstein manifolds collapsing to lower
dimensional metric spaces without \emph{a priori} curvature bounds --- for
instance, the collapsing of Ricci flat $K3$ surfaces constructed by Gross and
Wilson \cite{GW00} and recently by Hein, Sun, Viaclovsky and Zhang
\cite{HSVZ18} --- one wonders if there should be any sort of extension of
Fukaya's Key Lemma to Einstein manifolds that are pointed Gromov-Hausdorff
close to lower dimensional metric spaces at a given scale, \emph{without}
assuming (\ref{eqn: regularity}). To explain the basic setup in this situation,
let us focus on a small piece of a very collapsed Riemannian manifold with
almost non-negative Ricci curvature, and recall the following fundamental
theorem due to Cheeger and Colding (see \cite[Theorem 1.2]{ChCoII} and
\cite[Lemma 1.21]{ChNa14}):
\begin{theorem}[Cheeger-Colding's Almost Splitting Theorem]
\label{thm: almost_splitting} 
Let $(M^m,g)$ be a Riemannian manifold. There there exists $\varepsilon(m)>0$
and $l(m)>10$ to the following effect: suppose a geodesic $lr$-ball
$B(p,lr)\subset M$ satisfies $d_{GH}(B(p,lr),B^k(lr))\le \varepsilon r$ for some
integer $l>l(m)$ and some $\varepsilon \in (0,\varepsilon(m))$, where $B^k(lr)$
denotes the $k$-Euclidean $lr$-ball centered at the origin, then there is a
harmonic map $\Phi: B(p,4r)\to \mathbb{R}^k$ such that
\begin{enumerate}
  \item $\Phi(B(p,2r))\subset B^k(2r)$;
  \item $\sup_{B(p,2r)}|\nabla \Phi^a|\ \le\ C(m)$;
  \item $\fint_{B(p,2r)}\left|\langle \nabla \Phi^a,\nabla
  \Phi^b\rangle-\delta^{ab}\right|\ \le\ \Psi(\varepsilon,l^{-1}|m)$; and
  \item $\fint_{B(p,2r)}|Hess_{\Phi}|^2\ \le\ \Psi(\varepsilon,l^{-1}|m)$.
\end{enumerate}
for some $C(m)>0$ and $\Psi(\varepsilon,l^{-1}|m)>0$, with
$\Psi(\varepsilon,l^{-1}|m)\to 0$ as
$\varepsilon\to 0$ and $l\to \infty$. Here $\delta^{ab}$ denotes the
Kronecker delta and $\Phi^a$ ($a=1,\ldots,k$) denotes a component function of
the vector valued harmonic map $\Phi$.
\end{theorem}
When $k<m$, the map $\Phi$ resembles, in a certain sense, the fibration in the
case of collapsing with bounded sectional curvature. However, $\Phi$ is far from
being even a topological fibration, since it may well have singular values.
Considering each point $x\in \Phi^{-1}(\mathcal{R})$, with $\mathcal{R}\subset
\Phi(B(p,2r))$ denoting the regular values of $\Phi$, we define for any
function $u\in C^{\infty}(B(p,2r))$, the vector $\nabla^Tu(x)\in T_xM$ as the
part of $\nabla u(x)$ tangential to the fiber $\Phi^{-1}(\Phi(x))$. In contrast
to the case of collapsing with bounded curvature, even over the regular values
of $\Phi$, we have no information about the specific structure of the fibers of
$\Phi$, except their being closed and embedded submanifolds in $B(p,3r)$. The
purpose of this note is to show that even in this very rough case, there is
still an analogue of Fukaya's Key Lemma:
\begin{theorem}[Smallness of the Fiberwise Gradient $L^2$-Average]\label{thm:
main} Let $B(p,lr)$ be a geodesic ball in an $m$-dimensional Ricci flat
manifold, and suppose that $d_{GH}(B(p,lr),B^k(lr))\le \varepsilon r$ for some
$l>l(m)$ and $r\in (0,\varepsilon(m))$, where $l(m), \varepsilon(m)$ are
positive dimensional constants determined by Theorem~\ref{thm:
almost_splitting}. Then for any $u \in C^{\infty}(B(p,2r))$ satisfying $\Delta
u=\theta u$ for some $\theta \in \mathbb{R}$, we have
\begin{align}\label{eqn: main}
r\|\nabla^T u\|_{\bar{L}^2(B(p,r))}\ \le\
C(m,k,\theta)\|u\|_{L^{\infty}(B(p,2r))}\left(\varepsilon^{\frac{1}{2}}
+\Psi(\varepsilon,l^{-1}|m)\right).
\end{align}
Here the constant $C(m,k,\theta)>0$ is determined only by $m, k$ and
$\theta$, and $\|\cdot\|_{\bar{L}^2(B(p,2r))}$ denotes the $L^2$-average of a 
function on $B(p,2r)$.
\end{theorem}
\begin{remark}
Our estimate here only relies on the lower bound of the Ricci curvature. The
Ricci flatness guarantees that the metric is real analytic, and the theorem
still holds for any manifold with an \emph{analytic} metric, whose Ricci
curvature being bounnded below; see Theorem~\ref{thm: Ricci_tangential_L2}.
Especially, this theorem works for K\"ahler manifolds with holomorphic
bisectional curvature bounded below.
\end{remark}

Notice that the definition of $\nabla^Tu$ is not canonical --- it depends on the
harmonic almost splitting map $\Phi$. The same phenonmenon even occures in the
case of collapsing with bounded curvature, where the fibration map also faces
many choices, see \cite{Fukaya87ld}. 

We would now like to mention our novelty in proving Theorem~\ref{thm:
main} and briefly forcast the contents of the note. The conventional strategy in
dealing with collapsing when (\ref{eqn: regularity}) is satisfied, is to pull
the metric back to a local universal covering space, which approximately looks
like a product, and Fukaya's proof of Theorem~\ref{thm: Fukaya} hinges upon
this property. When (\ref{eqn: regularity}) is weakened to a mere Ricci
curvature lower bound, while this strategy works in certain situations like 
controlling the fundamental group or understanding the infinitesimal behavior of
the metric (see e.g. \cite{NZ14} and \cite{HKRX18}), it cannot lead us to the
more delicate gradient estimate as (\ref{eqn: main}), without any extra
assumption (like Ricci bounded covering geometry). 
This is mainly due to the lack of a fibration (or submersion) structure, as well
as the absence of the structural information of the $\Phi$ fibers. We realize,
however, that the specific structure of a $\Phi$ fiber do not affect our
estimate on the change of a function along the fiber --- only size
matters. Therefore, our main efforts are plunged into the immersion side of the
picture: the regular fibers of the almost \emph{submersion} $\Phi$ are compactly
\emph{embedded} submanifolds. In \S 2, we will investigate the variation of
$|\nabla^Tu|^2$ along the dynamics driven by $\nabla^Tu$, and obtain the estimate
\begin{align*}
\left|\nabla^Tu\right|^2(\gamma_x(t))\ \ge\ e^{-Ct}\left|\nabla^Tu\right|^2(x),
\end{align*}
where $\gamma_x$ is the flow line of the vector field $\nabla^Tu$ along
$\Phi^{-1}(\Phi(x))$, starting from $x$ with $\Phi(x)\in \mathcal{R}$. On
the other hand, for any $t>0$, $\gamma_x(t)$ stays within $\Phi^{-1}(\Phi(x))$, 
so the \emph{a priori} gradient estimate of $u$ and the diameter bound 
$\diam(\Phi^{-1}(\Phi(x)),g)< 2\varepsilon r$ ensure that
$|u(x)-u(\gamma_x(t))|<C\varepsilon$. Therefore, $|\nabla^Tu|(x)$ cannot be too
large compared to $\varepsilon$, because integrating the above inequality with
respect to $t$ controls $\left|\nabla^Tu\right|(x)$ from above. In \S 3, this
argument will be extended, in the $L^2$-average sense, across all fibers of
$\Phi$, with the integral Jacobian and Hessian estimates porvided by
Theorem~\ref{thm: almost_splitting}. Here in the process of extension, we need
the analyticity of the metric to ensure that the singular fibers have zero
total measure.

The estimate in Theorem~\ref{thm: main} is an effective $L^2$-average gradient
estimate rather than the original $L^{\infty}$-gradient estimate --- replacing
the pointwise estimates by the local $L^2$-average estimate as we change from
sectional curvature bounds to the corresponding Ricci curvature bounds is a
natural and necessary phenomenon ever since the early works of Colding on the
local $L^2$-average estimates of the regularized angle and distance functions,
see \cite{ColdingI, ColdingII, ColdingIII}. Such $L^2$-averative gradient
estimates usually suffice to relate the metric measure properties of the
collapsing manifolds with those of the collapsing limit spaces.

In \cite{ChCoIII}, Cheeger and Colding generalized Fukaya's eigenvalue
continuity theorem to manifolds with only Ricci curvature lower bound, without
referring to an analogue of Fukaya's Key Lemma, but we believe the $L^2$-average
tangential gradient estimate in Theorem~\ref{thm: main} may provide some new
tool to sharpen our understanding of the collapsing limits. Let us mention one
potential application of Theorem~\ref{thm: main} as an example: suppose a
sequence of Einstein manifolds $\{(M_i,g_i)\}$ of uniformly bounded diameters
and Einstein constants Gromov-Hausdorff converges to a limit metric space
$(X,d)$ of lower diemsion (in the sense of \cite[Theorem 1.12]{CoNa11}), then
by using a covering argument, as well as \cite[Theorem 3.23]{ChCoIII}, we can
apply Theorem~\ref{thm: main} to show that the eigenfunctions converge to those
on the limit in the $H^1$-sense (see \cite{Cheeger99} and \cite{ChCoIII}),
whence the $H^1$-convergence of the related heat kernels (see \cite[\S
6]{ChCoIII} and \cite{Ding02}), and consequently, this will lead to an
embedding theorem of $(X,d)$ into the infinite dimensional Hilbert space
$L^2(X)$, in view of the related arguments in \cite{Fukaya89}. Yet we will not
put more details here, since the above mentioned $H^1$-convergence and
embedding results have recently been obtained by Ambrosio, Honda, Portegies and
Tewodrose \cite{AHPT18} in more general settings. Their approach relies on the
more abstract theory of $RCD^{\ast}$-spaces developed by Ambrosio, Gigli and
Savar\'e, see \cite{AGS14a, AGS14b}; and this is in turn based on the
Lott-Sturm-Villani characterization of the Ricci curvature lower bound via
optimal transport, see\cite{LV09, Sturm06a, Sturm06b}.

\section{A priori estimate along the regular fibers}
\addtocontents{toc}{\protect\setcounter{tocdepth}{1}}

In this section we control the size of the tangential derivatives of a function
on a regular fiber of $\Phi$. Here we will overcome the main technical difficulty
--- the lack of a geometric structure of any regular fiber of $\Phi$ --- by
investigated the dynamics driven by the tangential gradient vector fields along
the regular fibers, only relying on the fact that they are small compactly
embedded sub-manifolds.

Let $\Phi: B(p,4r)\to \mathbb{R}^k$ be given as in Theorem~\ref{thm:
almost_splitting}. Since we may assume that $\Phi$ provides an $\varepsilon
r$-Gromov-Hausdorff approximation of $B(p,2r)$ to $B^k(2r)$ (see \cite[Lemma
9.16]{Cheeger10}), we have
\begin{align}\label{eqn: small_fiber}
\forall \vec{v}\in \Phi(B(p,4r)),\quad \diam (\Phi^{-1}(\vec{v}),g)\ \le\
2\varepsilon r.
\end{align}
 
Moreover, let us denote the Jacobian matrix of $\Phi$ as
\begin{align}\label{eqn: Jacobian_defn} \forall x\in B(p,r),\quad 
J_k(\nabla \Phi)(x)\ :=\ \left[\langle \nabla \Phi^a(x),\nabla
\Phi^b(x)\rangle\right]_{k\times k},
 \end{align}
with $\lambda(x)$ and $\Lambda(x)$ denoting its least and largest
 eigenvalues, respectively.
 
 For any smooth function $u$ on $B(p,2r)$, we now present the following \emph{a
 priori} estimate of the tangential derivative in terms of $\lambda$, $\Lambda$,
 $|Hess_u|$ and $|Hess_{\Phi}|$:
 \begin{lemma}[\emph{A priori} estimate]\label{lem:  apriori_estimate} 
 Fix $r\in (0,1)$. Let $\vec{v}\in \Phi(B(p,2r))\subset \mathbb{R}^k$ be a
 regular value of $\Phi$, and we employ the following notations:
\begin{align}\label{eqn: Phi_bounds}
\lambda\ :=\ \inf_{\Phi^{-1}(\vec{v})}\lambda(x),\quad
\Lambda\ :=\ \sup_{\Phi^{-1}(\vec{v})} \Lambda(x),\quad
\text{and}\quad C_0\ :=\ \max_{1\le b\le k}
\sup_{\Phi^{-1}(\vec{v})}r^2\left|Hess_{\Phi^b}\right|^2,
\end{align}
and
\begin{align}\label{eqn: u_bounds}
K\ :=\ \sup_{B(\Phi^{-1}(\vec{v}),2\varepsilon r)}\left(r|\nabla u|+r^2
\left|Hess_u\right|\right).
\end{align}
Let us also recall that $\nabla^Tu$ denotes the part of $\nabla u$ tangential to
$\Phi^{-1}(\vec{v})$, then
\begin{align}\label{eqn: tangential_u_bound}
\sup_{\Phi^{-1}(\vec{v})}r\left|\nabla^Tu\right|\ \le\ 
2\left(1+\lambda^{-1}\sqrt{\Lambda k}C_0\right)^{\frac{1}{2}}\
K\sqrt{\varepsilon}.
\end{align}
\end{lemma}

The basic idea in proving this lemma resembles that of the Key Lemma in
\cite[(4.3)]{Fukaya87b}: with the Hessian bound of the given function, its
gradient cannot alter too much within a small fiber; therefore the large
gradient at one point will create a marked difference from nearby points, in
terms of the value of the function; but this difference is in turn controlled by
the uniform gradient bound and the size of the fiber. However, in
\cite[(4.3)]{Fukaya87b}, such difference is captured along a geodesic which is
almost tangential to the fiber (by \cite[Lemma 4-8]{Fukaya87ld}), whereas in our
case, there is no such geometric structure available, and instead, we follow
the flow lines of the tangential gradient fields to compute the difference.

\begin{proof}
For the given $x\in\Phi^{-1}(\vec{v})$, let us denote
$\sup_{\Phi^{-1}(\vec{v})}\left|\nabla^Tu\right|=\delta_0$. We may assume
$\delta_0>0$ since otherwise (\ref{eqn: tangential_u_bound}) is trivial. 

Since $\vec{v}$ is a regular value of $\Phi$ and $\Phi$ is continuous, we know
that $\Phi^{-1}(\vec{v})$ is a closed embedded submanifold of $B(p,3r)$ (see
\cite[Theorem 1.38]{Warner}), and by (\ref{eqn: small_fiber}) we know that
$\Phi^{-1}(\vec{v})$ is compact, implying that $\lambda>0$. Moreover, there
exists some $x\in \Phi^{-1}(\vec{v})$ such that $\delta_0 =\left|
\nabla^Tu\right|(x)$, and we could consider the integral curve $\gamma_x$ of
$\nabla^Tu$ with initial value $x\in \Phi^{-1}(\vec{v})$. The curve
$\gamma_x(t)$ is defined at least up to some small $t>0$, and for any such $t$
we have
\begin{align*}
\frac{\text{d}}{\text{d}t}\left|\nabla^Tu\right|^2(\gamma_x(t))\ 
=\ \nabla_{\nabla^Tu}\left|\nabla^Tu\right|^2\ =\ 2\langle
\nabla_{\nabla^Tu}\nabla^Tu,\nabla^Tu\rangle.
\end{align*}

We also notice that for any smooth vector fields $X$ and $Y$
tangential to $\Phi^{-1}(\vec{v})$,
\begin{align*}
\nabla_X\langle \nabla^Tu,Y\rangle\ =\ &\nabla_X\langle \nabla
u,Y\rangle\\
=\ &Hess_u(X,Y)+\langle \nabla u,(\nabla_XY)^T\rangle+\langle \nabla
u,(\nabla_XY)^{\perp}\rangle\\
=\ &Hess_u(X,Y)+\langle \nabla^Tu,\nabla_XY\rangle-
\sum_{a,b=1}^{k}J_k(\nabla \Phi)^{-1}_{ab}\langle\nabla u,\nabla
\Phi^a\rangle \langle \nabla \Phi^b,\nabla_XY\rangle,
\end{align*}
where $(\nabla_XY)^T$ and $(\nabla_XY)^{\perp}$ are respectively the parts of
$\nabla_XY$ tangential and perpendicular to $\Phi^{-1}(\vec{v})$, and
$J_k(\nabla \Phi)^{-1}_{ab}$ is $(a,b)$-entry of the inverse of the Jacobian
matrix of $\Phi$.

Notice that $\langle \nabla \Phi^b,\nabla_XY\rangle$ is the $|\nabla \Phi^b|$
multiple of the second fundamental form of $\Phi^{-1}(\vec{v})$ in the direction
of $\nabla \Phi^b$, and we have
\begin{align*}
\langle \nabla \Phi^b,\nabla_XY\rangle\ =\ -Hess_{\Phi^b}(X,Y).
\end{align*}

Therefore we have
\begin{align}\label{eqn: tangential_u_derivative}
\begin{split}
\langle\nabla_X (\nabla^Tu),Y\rangle\ =\ &\nabla_X\langle\nabla^Tu,Y\rangle-
\langle \nabla^Tu,\nabla_XY\rangle\\
=\ &Hess_u(X,Y)+\sum_{a,b=1}^kJ_k(\nabla \Phi)^{-1}_{ab}\langle \nabla u,\nabla
\Phi^a\rangle Hess_{\Phi^b}(X,Y),
\end{split}
\end{align}
and by (\ref{eqn: Phi_bounds}) we could estimate
\begin{align*}
\left|\sum_{a,b=1}^kJ_k(\nabla \Phi)^{-1}_{ab}\langle \nabla u,\nabla
\Phi^a\rangle Hess_{\Phi^b}(X,Y)\right|\ \le\
\lambda^{-1}\sqrt{\Lambda k}C_0 r^{-1}|\nabla u||X||Y|.
\end{align*}
Further considering (\ref{eqn: u_bounds}), we obtain from (\ref{eqn:
tangential_u_derivative}) and the last inequality that
\begin{align}\label{eqn: tangential_u_derivative_bd}
\sup_{\Phi^{-1}(\vec{v})}\left|\nabla^T(\nabla^Tu)\right|\
\le\ (1+\lambda^{-1}\sqrt{\Lambda k}C_0) Kr^{-2}.
\end{align}
Especially, for $X=Y=\nabla^Tu$, since $\gamma_x$ is the integral curve of
$\nabla^Tu$, we have
\begin{align}\label{eqn: derivative}
\begin{split}
\left|\frac{\text{d}}{\text{d}t}\left|\nabla^Tu\right|^2(\gamma_x(t))\right|\
\le\ 2(1+\lambda^{-1}\sqrt{\Lambda k}C_0)
Kr^{-2}\left|\nabla^Tu\right|^2(\gamma_x(t))
\end{split}
\end{align}
whenever the curve $\gamma_x(t)$ is defined up to $t>0$.

Now integrating along $\gamma_x$ up to the time $t>0$ when $\gamma_x(t)$
remains being defined, we have
\begin{align}\label{eqn: gamma_dot_lb}
\begin{split}
\langle \nabla^Tu,\dot{\gamma}_x(t)\rangle\ =\
&\left|\nabla^Tu\right|^2(\gamma_x(t))\\
\ge\ &e^{-2(1+\lambda^{-1}\sqrt{\Lambda k}C_0) Kr^{-2}t}\delta_0^2.
\end{split}
\end{align}

This inequality tells that $\left|\nabla^Tu\right|^2(\gamma_x(t))$ is always
comparable to its initial value $\delta_0^2$, and it helps us extend the the
flow line $\gamma_x(t)$, i.e. we have the following
\begin{claim}\label{clm: extending_flow_line}
$\gamma_x(t)$ is defined for any $t\ge 0$.
\end{claim}

\begin{proof}[Proof of claim]
Clearly, $\gamma_x(t)$ exists at least up to some small positive time. Now let
$T$ be the supremum of the extence time for $\gamma_x(t)$ and assume, for the
purpose of a contradiction argument, that $T<\infty$. For any sequence
$t_i\nearrow T$, since $\{\gamma_x(t_i)\}\subset \Phi^{-1}(\vec{v})\Subset
B(x,3\varepsilon r)$ and $\Phi$ is continuous, there exists some $y\in
\Phi^{-1}(\vec{v})$ and a subsequence, still denoted by $\{\gamma_x(t_i)\}$,
such that $\lim_{i\to \infty}d_M(\gamma_x(t_i),y)= 0$.

We need to show that the above convergence is also with respect to the intrinsic
metric $d_{\Phi^{-1}(\vec{v})}$ --- this is because (\ref{eqn:
tangential_u_derivative_bd}) only provides a derivative control of $\nabla^Tu$
in the directions tangential to $\Phi^{-1}(\vec{v})$, and in order to apply this
to guarantee that (\ref{eqn: gamma_dot_lb}) persists in taking limit, we need to
show that the convergence $\gamma_x(t_i)\to y$ actually takes place within the
Riemannian manifold $(\Phi^{-1}(\vec{v}), g|_{\Phi})$, where $g|_{\Phi}$ denotes
the metric tensor $g$ restricted to $T\Phi^{-1}(\vec{v})$.

To this end, we notice that since $\vec{v}$ is a regular value of $\Phi$, there
is a positive radius $r_0\le \varepsilon r$, such that the following conditions
are satisfied for the geodesic ball $B(y,r_0)$:
\begin{enumerate}
  \item[(a)] there is a co-ordinate chart of $B(y,r_0)\subset B(p,r)$ in which
  $B(y,r_0)\cap \Phi^{-1}(\vec{v})$ is a single slice, see \cite[Theorem
  1.38]{Warner};
  \item[(b)] $B(y,r_0)$ is contained in a normal neightborhood of $y$ ---
  especially, within $B(y,r_0)$ the distance $d_M(y,\cdot)$ is realized by the
  length of $g$-geodesic segments emanating from $y$ and entirely contained in
  $B(y,r_0)$, see \cite[Proposition 3.6]{doCarmo}.
  \end{enumerate} 
Now regarding $y$ as a point on the Riemannian manifold $(\Phi^{-1}(\vec{v}),
g|_{\Phi})$, and let $r_1\le r_0$ be a radius such that
$B_{d_{\Phi^{-1}(\vec{v})}}(y,2r_1)$ is contained in a normal neighborhood of
$y$ for $g|_{\Phi}$. Now there is a radius $r_2\in (0,r_1)$ so that within
$B(y,r_2)\cap \Phi^{-1}(\vec{v})$, the intrinsic distance
$d_{\Phi^{-1}(\vec{v})}(y,\cdot)$ is realized by the length of
$g|_{\Phi}$-geodesic segments emanating from $y$ and staying within
$B_{d_{\Phi^{-1}(\vec{v})}} (y,r_1) \subset \Phi^{-1}(\vec{v})$ entirely. Since
under the local co-ordinate chart in (a) any such $g|_{\Phi}$-geodesic segment
stays within the slice, we can obviously see that the intrinsic distance
$d_{\Phi^{-1}(\vec{v})}(y,\cdot)$ and the extrinsic distance
$d_M(y,\cdot)|_{\Phi^{-1}(\vec{v})}$ are comparable --- up to a factor
controlled by $\|\Phi\|_{C^3(B(y,4\varepsilon r))}$. Therefore, since each 
$\gamma(t_i)\in \Phi^{-1}(\vec{v})$, the convergence
$\lim_{i\to \infty}d_M(y,\gamma_x(t_i))= 0$ is equivalent to the convergence
$\lim_{i\to \infty}d_{\Phi^{-1}(\vec{v})}(\gamma_x(t_i),y)= 0$ .

Now we see that (\ref{eqn: tangential_u_derivative_bd}) and (\ref{eqn:
gamma_dot_lb}) guarantee that $\nabla^Tu(\gamma(t_i))\to \nabla^Tu(y)$ as
$i\to\infty$ in $T\Phi^{-1}(\vec{v})$, and consequently
$\left|\nabla^Tu\right|^2(y)\ge e^{-2(1+\lambda^{-1}\sqrt{\Lambda k}C_0)
Kr^{-2}T}\delta_0^2>0$; on the other hand, since $B(y,r_2)\cap
\Phi^{-1}(\vec{v})$ is a slice in $B(y,r_1)$, it is easy to see that we can
smoothly extend $\gamma_x$ beyond $T$ with the initial value $y$, in the
direction of $\nabla^Tu(y)$ --- breaking the supposed supremum of the existence
time.
\end{proof}

Continuing our discussion, we could now integrate the inequality above for all
$t>0$:
\begin{align}\label{eqn: value_change}
\begin{split}
u(\gamma_x(t))-u(x)\ =\ &\int_{0}^t\langle \nabla u, 
\dot{\gamma}_x(s)\rangle\ \text{d}s\\
=\ &\int_{0}^t\langle \nabla^T u, 
\dot{\gamma}_x(s)\rangle\ \text{d}s\\
\ge\
&\frac{1-e^{-2(1+\lambda^{-1}\sqrt{\Lambda k}C_0) Kr^{-2}t}}{
2(1+\lambda^{-1}\sqrt{\Lambda k}C_0) Kr^{-2}} \delta_0^2.
\end{split}
\end{align}

On the other hand, since $\Phi^{-1}(\vec{v})\subset B(x,4\varepsilon r)$, by
(\ref{eqn: small_fiber}) and (\ref{eqn: u_bounds}) we see that 
\begin{align}\label{eqn: fiber_size}
\forall y,y'\in \Phi^{-1}(\vec{v}),\quad |u(y)-u(y')|\ \le\ Kr^{-1} d_g(y,y') \ 
\le\ 2K\varepsilon.
\end{align}

Since $\gamma_x$ is the integral curve of a vector field tangent to
$\Phi^{-1}(\vec{v})$, $\gamma_x(t)\in \Phi^{-1}(\vec{v})$ for all $t\ge 0$,
combining (\ref{eqn: value_change}) and (\ref{eqn: fiber_size}) we get
the following upper bound for $t$:
\begin{align}\label{eqn: time_ub}
t\ \le\ \frac{\ln \delta_0^2-\ln\left(\delta_0^2
-4(1+\lambda^{-1}\sqrt{\Lambda k}C_0) K^2\varepsilon
r^{-2}\right)}{2(1+\lambda^{-1} \sqrt{\Lambda k}C_0) Kr^{-2}}.
\end{align} 

However, since we know (\ref{eqn: value_change}) is valid for all $t\ge 0$, the
right-hand side of (\ref{eqn: time_ub}) must be $\infty$, that is to say,
we have
\begin{align*}
\delta_0\ \le\ 2\left(1+\lambda^{-1}\sqrt{\Lambda k}C_0\right)^{\frac{1}{2}}
K\sqrt{\varepsilon}r^{-1}.
\end{align*}
This is the desired derivative bound.
\end{proof}

\begin{remark}\label{rmk: key_point}
In the proof of the \emph{a priori} estimate above, we notice that as long as
$B(p,2r)\subset M$ as smooth manifolds, and $\Phi:B(p,2r)\to \mathbb{R}^k$ is a
smooth map whose regular fibers are bounded in $B(p,2r)$, then the tangential
flow lines of a $C^2$ function $u$ is always defined for any $t>0$ within a
regular $\Phi$-fiber --- this is guaranteed, via Claim~\ref{clm:
extending_flow_line}, by the \emph{local} bounds of $J_k(\nabla \Phi)$,
$Hess_u$ and $Hess_{\Phi}$ around the fiber, and such bounds are in turn
guaranteed by the smoothness of $M$ and $\Phi$, and the compactness and
regularity of the fiber in consideration. The key point is that we \emph{do
not} need to assume any uniform bound on $J_k(\nabla \Phi)$, $Hess_u$ or
$Hess_{\Phi}$ to conclude the long-time existence of the dynamics driven by
$\nabla^Tu$ on any regular $\Phi$-fiber, and therefore we are free to
differentiate and integrate along the flows lines of $\nabla^Tu$.
\end{remark}

\section{The $L^2$-Average gradient control of the tangential derivatives}
In this section we prove our main estimate, Theorem~\ref{thm: main}. We will
begin with extending the fiber-wise estimate in Lemma~\ref{lem:
apriori_estimate} across all regular fibers of $\Phi$. In fact, once we know the
long-time existence of the flow lines of $\nabla^Tu$ on the regular fibers, an
integral version of the argument employed in Lemma~\ref{lem: apriori_estimate}
enables us to weaken the assumptions on $|Hess_u|$ and $|Hess_{\Phi}|$ to more
natural integral bounds:
\begin{prop}[Interior $L^2$ estimate for the tangential
gradients]\label{prop: interior_L2_estimate}
Fix $r\in (0,1)$ and $C_0\in (0,1)$. Let $B(p,4r)\subset M$ be a geodesic ball
in a smooth $m$ dimensional Riemannian manifold $(M,g)$. With the given $\Phi$
defined on $B(p,4r)$ as before, assuming besides (\ref{eqn: small_fiber}), that
$u\in C^{\infty}(B(p,2r))$ satisfies
\begin{align}\label{eqn: u_W22_bounds}
\sup_{B(p,2r)}\left(|u|^2+r^2|\nabla u|^2\right)
+r^4\fint_{B(p,2r)}\left|Hess_u\right|^2\ \le\ K^2,
\end{align}
and that $\Phi$ satisfies the estimates 
\begin{align}\label{eqn: Phi_W22_bounds}
\sup_{B(p,4r)}\max_{1\le a\le k}|\nabla \Phi^a|\ \le\ 1+C_0\quad\text{and}\quad
r^2\sum_{b=1}^k \fint_{B(p,4r)} \left|Hess_{\Phi^b}\right|^2\ \le\ C_0^2,
\end{align} 
then for $C_1=C_1(C_0,k)= 8k^2(1+C_0)^{k-1}$, we have
\begin{align}\label{eqn: interior_L2}
r^2\int_{\Phi^{-1}(\mathcal{R})}\left|\nabla^Tu\right|^2|J_k(\nabla \Phi)|\
\dvol_g\ \le\ C_1|B(p,r)|K^2\left(\sqrt{\varepsilon}+C_0\right),
\end{align}
where $|J_k(\nabla \Phi)|=\sqrt{\det J_k(\nabla \Phi)}$ is the Jacobian of the
map $\Phi$ (see (\ref{eqn: Jacobian_defn})), and $\mathcal{R}$ denotes the
regular values of $\Phi$ in $\Phi(B(p,r-4\varepsilon r))$.
\end{prop}

\begin{proof}
 First notice that since $\mathcal{R}$ consists of regular values of
 $\Phi$ in $\Phi(B(p,r-4\varepsilon r))$, we must have
 $\Phi^{-1}(\mathcal{R})\subset B(p,r)$ by (\ref{eqn: small_fiber}). Moreover,
 each $\Phi$-fiber over $\mathcal{R}$ is a compact, embedded sub-manifold of
 $B(p,r)$. By the observation just discussed in Remark~\ref{rmk: key_point}, we
 know that the flow lines $\gamma_x(t)$ of $\nabla^Tu$, with initial values
 $x\in \Phi^{-1}(\mathcal{R})$ exists for all $t\ge 0$. It is also clear that
 $\Phi^{-1}(\mathcal{R})$, as a union of regular $\Phi$-fibers, each of which
 being invariant under the diffeomorphism (of the fiber) generated by
 $\nabla^Tu$, is itself invariant under the evolution driven by $\nabla^Tu$.
 
By (\ref{eqn: tangential_u_derivative}) we could compute for any $t>0$ to see
\begin{align}\label{eqn: integral_derivative}
\begin{split}
&\frac{\text{d}}{\text{d}t}\int_{\Phi^{-1}(\mathcal{R})}
\left(\left|\nabla^Tu\right|^2|J_k(\nabla \Phi)|\right)(\gamma_x(t))\
\dvol_g(x)\\
=\quad &\int_{\Phi^{-1}(\mathcal{R})} \left(\frac{\text{d}}{\text{d}t}
\left|\nabla^Tu\right|^2|J_k(\nabla \Phi)|+\left|\nabla^Tu\right|^2
\frac{\text{d}}{\text{d}t}|J_k(\nabla \Phi)| \right)(\gamma_x(t))\ 
\dvol_g(x)\\
=\quad &2\int_{\Phi^{-1}(\mathcal{R})}
 Hess_u(\dot{\gamma}_x(t),\dot{\gamma}_x(t))|J_k(\nabla \Phi)|(\gamma_x(t))\
 \dvol_g(x)\\
&-2\sum_{a,b=1}^{k}\int_{\Phi^{-1}(\mathcal{R})}
\left(J_k(\nabla\Phi)_{ab}^{-1} \langle \nabla \Phi^a,\nabla
u\rangle Hess_{\Phi^b}(\nabla^Tu,\nabla^Tu) |J_k(\nabla \Phi)|\right)
(\gamma_x(t))\ \dvol_g(x)\\
&+\sum_{a,b=1}^k\int_{\Phi^{-1}(\mathcal{R})}
\left(\left|\nabla^Tu\right|^2 J_k(\nabla \Phi)_{ab}^{-1}
 Hess_{\Phi^a}(\nabla^Tu,\nabla \Phi^b) |J_k(\nabla \Phi)|\right)(\gamma_x(t))\
\dvol_g(x).
\end{split}
\end{align}

In order to obtain a uniform estimate only depending on (\ref{eqn:
u_W22_bounds}) and (\ref{eqn: Phi_W22_bounds}), we need to analyze the following
(multi-) linear quantites 
\begin{align}\label{eqn: key_terms}
\begin{cases}
F(\nabla u,\nabla^Tu,\nabla^Tu)\ :=\ \sum_{a,b=1}^kJ_k(\nabla
\Phi)_{ab}^{-1}\langle \nabla \Phi^a,\nabla u\rangle
Hess_{\Phi^b}(\nabla^Tu,\nabla^Tu)|J_k(\nabla \Phi)|\\ \\
G(\nabla^Tu)\ :=\ \sum_{a,b=1}^kJ_k(\nabla
\Phi)_{ab}^{-1}Hess_{\Phi^a}(\nabla^Tu,\nabla \Phi^b) |J_k(\nabla \Phi)|
\end{cases}
\end{align}
at each $y=\gamma_x(t)\in \Phi^{-1}(\mathcal{R})$. We will rely on the
invariance of these sum under the orthogonal transformations.

To this end, fix some $y=\gamma_x(t)\in \Phi^{-1}(\mathcal{R})$, and let
$Q=[q_{ab}]\in O(k)$ be an orthogonal matrix that diagonalizes $J_k(\nabla
\Phi)(y)$: the specific value of $Q$ depends on $y\in \Phi^{-1}(\mathcal{R})$
but $Q$ is a \emph{constant} matrix. We denote $\phi^a
=\sum_{b=1}^kq_{ab}\Phi^b$, and $\phi=[\phi^1,\ldots, \phi^k]^t$ (the transpose
of the row vector), then
\begin{align*}
J_k(\nabla \phi)\ =\ 
\begin{bmatrix}
\nabla \phi^1\\  \vdots\\ \nabla \phi^k
\end{bmatrix}\cdot [\nabla \phi^1,\ldots,\nabla \phi^k]\ 
=\ Q\begin{bmatrix}\nabla \Phi^1\\ \vdots\\
\nabla \Phi^k\end{bmatrix}\cdot [\nabla \Phi^1,\ldots,\nabla \Phi^k]Q^t\
=\ QJ_k(\nabla \Phi)Q^t,
\end{align*}
where the dot ``\ $\cdot$\ '' denotes the inner product of vector fields with
respect to the metric tensor $g$. Notice that since $Q$ is a constant matrix and
each $\phi^b$ is a constant linear combanition of $\Phi^1,\ldots,\Phi^k$, the
covariant derivatives land directly on $\Phi^1,\ldots,\Phi^k$ --- besides
$\nabla \phi^a=\sum_{a=1}^kq_{ab}\nabla \Phi^b$, we also have
$Hess_{\phi^a}=\sum_{b=1}^kq_{ab}Hess_{\Phi^b}$.

Now suppose $J_k(\nabla \phi)(y)$ has eigenvalues $\lambda_1\le \lambda_2\le
\ldots\le \lambda_k$ on the diagonal, then for each $a=1,\ldots,k$, we have
$\lambda_a=|\nabla \phi^a|^2(y)$. With the notation $c_a:=\langle \nabla
u,\frac{\nabla \phi^a}{|\nabla \phi^a|}\rangle$, we see $|c_a|\le Kr^{-1}$. Also
notice that by the invariance of the determinant, we have $\det J_k(\nabla
\Phi)=\det J_k(\nabla \phi)= \lambda_1\lambda_2\cdots \lambda_k$. Now by
orthogonality of $Q$ we can compute
\begin{align}\label{eqn: computation_1}
\begin{split}
&\sum_{a,b=1}^k J_k(\nabla \Phi)_{ab}^{-1}\langle \nabla\Phi^a,\nabla u\rangle
Hess_{\Phi^b}(\nabla^Tu,\nabla^Tu)|J_k(\nabla \Phi)|\\
=\quad &[\langle \nabla\Phi^1,\nabla u\rangle,\ldots,\langle \nabla
\Phi^k,\nabla u\rangle] Q^t 
\begin{bmatrix} \lambda_1^{-1}& &\\
  & \ddots&\\
  &       &\lambda_k^{-1}
\end{bmatrix}Q\begin{bmatrix}
Hess_{\Phi^1}(\nabla^Tu,\nabla^Tu)\\ \vdots \\
Hess_{\Phi^k}(\nabla^Tu,\nabla^Tu)
\end{bmatrix}\sqrt{\lambda_1\cdots \lambda_k}\\
=\quad &[\langle \nabla\phi^1,\nabla u\rangle,\ldots,\langle \nabla
\phi^k,\nabla u\rangle] \begin{bmatrix} \lambda_1^{-1}& &\\
 & \ddots&\\
 &   &\lambda_k^{-1}
\end{bmatrix} 
\begin{bmatrix}
Hess_{\phi^1}(\nabla^Tu,\nabla^Tu)\\ \vdots \\
Hess_{\phi^k}(\nabla^Tu,\nabla^Tu)
\end{bmatrix}\sqrt{\lambda_1\cdots \lambda_k}\\
=\quad &\sum_{a=1}^k c_a\lambda^{-\frac{1}{2}}_a
Hess_{\phi^a}(\nabla^Tu,\nabla^Tu)\sqrt{\lambda_1\cdots
\lambda_k}\\
=\quad
&\sum_{a=1}^kc_aHess_{\phi^a}(\nabla^Tu,\nabla^Tu)
\prod_{b\not=a}\lambda_b^{\frac{1}{2}}.
\end{split}
\end{align}
To estimate the last line above, we notice that all terms come with a
\emph{positive} power of $\lambda_1,\ldots,\lambda_k$ --- this is crucial for
us, as we \emph{do not} have any uniformly positive lower bound of
$\lambda_1$ when $y\in \Phi^{-1}(\mathcal{R})$ varies. We will also rely on the
assumption (\ref{eqn: Phi_W22_bounds}), the linearity of taking covariant
derivatives, as well as the fact that $Q\in O(k)$ to see that for each
$a=1,\ldots,k$,
\begin{align*}
\lambda^{\frac{1}{2}}_a\ =\ |\nabla \phi^a|(y)\ \le\ \max_{1\le b\le
k}|\nabla \Phi^b|(y)\ \le\ 1+C_0\quad 
\text{and}\quad \left|Hess_{\phi^a}\right|(y)\ \le\ \max_{1\le b\le k}
\left|Hess_{\Phi^b}\right|(y).
\end{align*}
Notice that the right-hand sides of the above inequalities are independent of
the specific matrix $Q\in O(k)$, and they are also independent of the choice of
$y\in \Phi^{-1}(\mathcal{R})$ ---  we obtain the following
\emph{uniform} estimate across $\Phi^{-1}(\mathcal{R})$:
\begin{align}\label{eqn: nabla_u_perp}
\begin{split}
\left|F(\nabla u, \nabla^Tu,\nabla^Tu)\right|\ \le\ 
k(1+C_0)^{k-1}K^3r^{-3}\sum_{b=1}^k\left|Hess_{\Phi^b}\right|.
\end{split}
\end{align}

In the same setting as above, we have at $y\in \Phi^{-1}(\mathcal{R})$ the
matrix $Q=[q_{ab}]\in O(k)$ to help compute
\begin{align}\label{eqn: computation_2}
\begin{split}
&\sum_{a,b=1}^kJ(\nabla \Phi)_{ab}^{-1}Hess_{\Phi^a}(\nabla^Tu,\nabla
\Phi^b) |J_k(\nabla \Phi)|\\
=\quad &\sum_{a,b=1}^k\sum_{c=1}^kq_{ca}\lambda_c^{-1}q_{cb}
 Hess_{\Phi^a}(\nabla^Tu,\nabla \Phi^b) \sqrt{\lambda_1\cdots \lambda_k}\\
=\quad &\sum_{a=1}\lambda^{-\frac{1}{2}}_a Hess_{\phi^a}(\nabla^Tu, \nabla
\phi^a)\prod_{b\not=a}\lambda_b^{\frac{1}{2}},
\end{split}
\end{align}
and therefore by $\lambda_a^{-\frac{1}{2}}|\nabla \phi^a|=1$ we could estimate
as before to see
\begin{align}\label{eqn: Jacobian_derivative}
\left|G(\nabla^Tu)\right|\ \le\
k(1+C_0)^{k-1}Kr^{-1}\sum_{b=1}^k\left|Hess_{\Phi^b}\right|,
\end{align}
as an estimate holds \emph{uniformly} across $\Phi^{-1}(\mathcal{R})$.

Recalling (\ref{eqn: integral_derivative}), we could now control the
integral by integrating the above estimates (\ref{eqn: nabla_u_perp}) and
(\ref{eqn: Jacobian_derivative}) across $\Phi^{-1}(\mathcal{R})$ as following:
\begin{align}\label{eqn: integral_derivative_lb1}
\begin{split}
&\frac{\text{d}}{\text{d}t}\int_{\Phi^{-1}(\mathcal{R})}
\left(\left|\nabla^Tu\right|^2|J_k(\nabla \Phi)|\right)(\gamma_x(t))\
\dvol_g(x)\\
\ge \quad &-2(1+C_0)^{k}K^2r^{-2}
\int_{\Phi^{-1}(\mathcal{R})}\left|Hess_u\right|(\gamma_x(t))\ \dvol_g(x)\\
&-3k(1+C_0)^{k-1}K^3r^{-3}\int_{\Phi^{-1}(\mathcal{R})}
\sum_{b=1}^k \left|Hess_{\Phi^b}\right|(\gamma_x(t))\ \dvol_g(x).
\end{split}
\end{align}
To bound the last two integrals, we notice that $\Phi^{-1}(\mathcal{R})$ is
invariant under the flow of $\nabla^Tu$, meaning that these integrals are the
same as the ones in (\ref{eqn: u_W22_bounds}) and (\ref{eqn: Phi_W22_bounds}),
restricted on the subset $\Phi^{-1}(\mathcal{R})$. Therefore we arrive at the
following lower bound:
\begin{align}\label{eqn: integral_derivative_lb2}
\frac{\text{d}}{\text{d}t}\int_{\Phi^{-1}(\mathcal{R})}
\left(\left|\nabla^Tu\right|^2|J_k(\nabla \Phi)|\right)(\gamma_x(t))\
\dvol_g(x)\ \ge\ -4k^2(1+C_0)^k|B(p,r)|K^3r^{-4}.
\end{align}

Integrating the above inequality in $t$, we see for any $t>0$ fixed,
\begin{align}\label{eqn: integral_derivative_lb3}
\begin{split}
&\int_{\Phi^{-1}(\mathcal{R})}\left(\left|\nabla^Tu\right|^2|J_k(\nabla
\Phi)|\right)(\gamma_x(t))\ \dvol_g(x)\\
\ge\quad &\int_{\Phi^{-1}(\mathcal{R})}\left|\nabla^Tu\right|^2(x)
|J_k(\nabla \Phi)|(x)\ \dvol_g(x) -4k^2(1+C_0)^k|B(p,r)| K^3 r^{-4}t.
\end{split}
\end{align}

Now we computate the variation of $\int u|J_k(\nabla \Phi)|$ driven by
$\nabla^Tu$:
\begin{align}\label{eqn: integral_variation_u}
\begin{split}
&\frac{\text{d}}{\text{d}t}\int_{\Phi^{-1}(\mathcal{R})}\left(u|J_k(\nabla
\Phi)|\right)(\gamma_x(t))\ \dvol_g(x)\\
=\quad &\int_{\Phi^{-1}(\mathcal{R})}\left(\left|\nabla^Tu\right|^2|J_k(\nabla
\Phi)|+uG(\nabla^Tu)\right) (\gamma_x(t))\ \dvol_g(x)
\end{split}
\end{align}
and by (\ref{eqn: u_W22_bounds}), (\ref{eqn: Phi_W22_bounds}), (\ref{eqn:
Jacobian_derivative}) and (\ref{eqn: integral_derivative_lb3}) we can estimate
\begin{align}\label{eqn: integral_variation_u_lb}
\begin{split}
&\frac{\text{d}}{\text{d}t}\int_{\Phi^{-1}(\mathcal{R})}\left(u|J_k(\nabla
\Phi)|\right)(\gamma_x(t))\ \dvol_g(x)\\
\ge\quad
&\int_{\Phi^{-1}(\mathcal{R})}\left|\nabla^Tu\right|^2(x)
|J_k(\nabla \Phi)|(x)\ \dvol_g(x) -4k^2(1+C_0)^k|B(p,r)| K^3
r^{-4}t\\
&-k(1+C_0)^{k-1}Kr^{-1} \int_{\Phi^{-1}(\mathcal{R})} u\sum_{b=1}^k
\left|Hess_{\Phi^b}\right|(\gamma_x(t))\ \dvol_g(x)\\
\ge\quad &\int_{\Phi^{-1}(\mathcal{R})}\left|\nabla^Tu\right|^2(x)
|J_k(\nabla \Phi)|(x)\ \dvol_g(x) -4k^2(1+C_0)^k|B(p,r)| K^3 r^{-4}t\\
&-k^2(1+C_0)^{k-1}|B(p,r)|K^2C_0r^{-2}.
\end{split}
\end{align} 

Integrating the last inequality in $t$ once again we see for any $t>0$,
\begin{align}\label{eqn: change_u_lb}
\begin{split}
&\int_{\Phi^{-1}(\mathcal{R})}\left(u |J_k(\nabla \Phi)|\right)(\gamma_x(t))\
\dvol_g(x)-\int_{\Phi^{-1}(\mathcal{R})}u(x)|J_k(\nabla \Phi)|(x)\ \dvol_g(x)\\
 \ge\quad  &t\int_{\Phi^{-1}(\mathcal{R})}\left|\nabla^Tu\right|^2(x)
 |J_k(\nabla\Phi)|(x)\ \dvol_g(x)-2k^2(1+C_0)^k |B(p,r)|K^3r^{-4}t^2\\
 &-k^2(1+C_0)^{k-1}|B(p,r)|K^2C_0r^{-2}t.
 \end{split}
 \end{align}

On the other hand, by the uniform gradient controls (\ref{eqn: u_W22_bounds})
and (\ref{eqn: Phi_W22_bounds}), we have $\left|\nabla^T u\right|\le Kr^{-1}$
and $|J_k(\nabla \Phi)|\le (1+C_0)^k$ on $\Phi^{-1}(\mathcal{R})$, and by the small
fiber assumption (\ref{eqn: small_fiber}), we know that
$d_g(\gamma_x(t),\gamma_x(0))\le 2\varepsilon r$ for any $t>0$ and initial value
$x\in \Phi^{-1}(\mathcal{R})$. Therefore
\begin{align*}
\forall x\in \Phi^{-1}(\mathcal{R}),\ \forall t>0,\quad
|u(\gamma_x(t))-u(x)||J(\nabla \Phi)|(\gamma_x(t))\ \le\
2(1+C_0)^kK\varepsilon.
\end{align*}

Moreover, since for any smooth curve $\gamma$, we have
\begin{align*}
\left|\nabla_{\dot{\gamma}}|J_k(\nabla \Phi)|\right|\ =\ |G(\dot{\gamma})|\ \le\
k(1+C_0)^{k-1}|\dot{\gamma}|\sum_{b=1}^k\left|Hess_{\Phi^b}\right|,
\end{align*}
which is obtained following the same path leading to (\ref{eqn:
Jacobian_derivative}). Now fixing some $t>0$, we see that almost every pair of
points $(x,\gamma_x(t))$ with $x\in \Phi^{-1}(\mathcal{R})$ could be connected
by a unique minimal geodesic $\sigma_{x,\gamma_x(t)}$ of speed
$d(x,\gamma_x(t))$, whose image is entirely contained in $B(p,r)$. Therefore,
integrating the inequality
\begin{align*}
\left||J_k(\nabla \Phi)|(\gamma_x(t))-|J_k(\nabla \Phi)|(x)\right|\ \le\
k(1+C_0)^{k-1}d(x,\gamma_x(t))\sum_{b=1}^k\int_0^1\left|Hess_{\Phi^b}\right|
(\sigma_{x,\gamma_{x}(t)}(s))\ \text{d}s
\end{align*} 
in $x\in \Phi^{-1}(\mathcal{R})$, we see that
\begin{align*}
&\int_{\Phi^{-1}(\mathcal{R})}\left||J_k(\nabla
\Phi)|(\gamma_x(t))-|J_k(\nabla \Phi)|(x)\right|\ \dvol_g(x)\\
 \le\quad
&2\varepsilon r k(1+C_0)^{k-1}\sum_{b=1}^k \int_{B(p,r)}
\left|Hess_{\Phi^b}\right|\\
\le\quad &2\varepsilon k(1+C_0)^{k}|B(p,r)|.
\end{align*}
Combining the above estimates we see for and fixed $t>0$ that
\begin{align}\label{eqn: change_integral_ub}
\begin{split}
&\left|\int_{\Phi^{-1}(\mathcal{R})}\left(u|J_k(\nabla
\Phi)|\right)(\gamma_x(t))\
\dvol_g(x)-\int_{\Phi^{-1}(\mathcal{R})}u(x)|J_k(\nabla \Phi)|(x)\
\dvol_g(x)\right|\\
 \le\quad &4k(1+C_0)^k|B(p,r)|K\varepsilon.
\end{split}
\end{align}

Now recalling the previous lower bound (\ref{eqn: change_u_lb}), we obtain the
following estimate for any $t\ge 0$:
\begin{align*}
&\int_{\Phi^{-1}(\mathcal{R})}\left|\nabla^Tu\right|^2 
|J_k(\nabla \Phi)|\ \dvol_g\\
\le\quad &\left(4t^{-1}\varepsilon +2tK^2r^{-4}\right)k^2(1+C_0)^k|B(p,r)|K
+k^2(1+C_0)^{k-1}|B(p,r)|K^2C_0r^{-2}.
\end{align*}

Now choosing $t=\sqrt{\varepsilon }K^{-1}r^2$, we immediately have
\begin{align*}
r^2\int_{\Phi^{-1}(\mathcal{R})}\left|\nabla^Tu\right|^2
|J_k(\nabla \Phi)|\ \dvol_g\ \le\
16k^2(1+C_0)^{k-1} |B(p,r)|K^2\left(\sqrt{\varepsilon}+C_0\right),
\end{align*}
whence the desired estimate (\ref{eqn: interior_L2}).
\end{proof}

\begin{remark}
It is necessary to integrate against $|J_k(\nabla \Phi)|$ in (\ref{eqn:
interior_L2}), and this is in correspondence with the measured Gromov-Hausdorff
convergence.
\end{remark} 

\begin{remark}
The lack of a uniform positive lower bound of $\lambda$ is also a crucial
problem even in the non-collapsing setting, see especially the Transformation
Thoerem of Cheeger and Naber~\cite[Theorem 1.32]{ChNa14}, which is the major
technical input of \cite{ChNa14} towards their solution of the Codimension Four
Conjecture. The invariance of the canonical quantities related to $J_k(\nabla
\Phi)$, under the action of the orthogonal group, has also been explored in
\cite{HW18} to understand the asymptotic behavior of $|J_k(\nabla \Phi)|$ on
complete non-compact manifolds with non-negative Ricci curvature.
\end{remark}

In order to apply our previous estimate in Proposition~\ref{prop:
interior_L2_estimate}, we will extend the domain of integral across the
singular fibers of $\Phi$, while still end up with the same estimates. As
discussed in the introduction, the situation we consider is a (volume)
collapsing sequence of $m$ dimensional Riemannian manifolds with a uniform
Ricci curvature lower bound.


Besides the desired estimates in Theorem~\ref{thm: almost_splitting} that the
$\Psi(\varepsilon)$-splitting map $\Phi$ satisfies, one key property is that
$\Phi$ is actually a \emph{harmonic map}. If the domain of $\Phi$ is
containted in an Einstein manifold, then around any point we could write down
the equations $\Delta_g \Phi^a=0$ ($a=1,\ldots,k$) under the harmonic
coordinates. Under such coordinates, the components of $g$ are real analytic and 
determine the coefficients of $\Delta_g$, implying that each $\Phi^a$
($a=1,\ldots,k$) is an analytic function. Therefore $\Phi$ is an analytic map
once we assume its domain is in an Einstein manifold.
To extend the estimate in Proposition~\ref{prop: interior_L2_estimate} across
the singular fibers of $\Phi$, let us recall the following fact for analytic
maps (see \cite[\S 3.1, Exercise 4(a)]{Hirsch}):
\begin{lemma}[Nullity of singular fibers]\label{lem: null_singular_fiber}
Suppose $\Phi:M\to N$ is an analytic map, whose domain is connected and has
dimension no less than that of the codomain. Let $\Sigma\subset M$ denote the
set of singular points of $\Phi$, then $\Phi^{-1}(\Phi(\Sigma))$ has measure
zero in $M$.
\end{lemma}


Applying a standard technique due to Cheeger and Colding we could further relax
the assumption of $\|Hess_u\|_{L^2}$ to a local integral bound of $\Delta u$ in
the following 
\begin{theorem}\label{thm: Ricci_tangential_L2} 
With the same setting as in Theorem~\ref{thm: almost_splitting}, we
further assume that $g$ is real analytic. Let $u$ be a smooth function on
$B(p,2r)$ with
\begin{align}\label{eqn: u_C1_bound}
\sup_{B(x,2r)}|u|+r|\nabla u|\ \le\ K,
\end{align}
then there is some positive constant $C_2(m,k)$, independent of
$\varepsilon\in (0,\varepsilon(m))$ and $r\in (0,1)$, such that
\begin{align}
r^2\fint_{B(p,r)}|\nabla^Tu|^2\ \dvol_g\ \le\
C_2(m,k)\left(\varepsilon^{\frac{1}{2}}+\Psi(\varepsilon,l^{-1}|m)\right)
\left(K^2+\|\Delta u\|_{\bar{L}^2(B(p,2r))}Kr^2\right),
\end{align}
where the $\bar{L}^2$ norm denotes the $L^2$-average, and 
$\Psi(\varepsilon,l^{-1}|m)$ is the same as the one obtained in
Theorem~\ref{thm: almost_splitting}.
\end{theorem}
\begin{proof}
Let $\mathcal{R}\subset \Phi(B(p,r))$ be the regular values of $\Phi$, then by 
Lemma~\ref{lem: null_singular_fiber} we clearly have
\begin{align}\label{eqn: full_measure}
|\Phi^{-1}(\mathcal{R})|\ =\ |B(p,r)|.
\end{align}
Moreover, the assumption (\ref{eqn: Phi_W22_bounds}) of $\Phi$ in
Proposition~\ref{prop: interior_L2_estimate} is satisfied with
$C_0=\Psi=\Psi(\varepsilon,l^{-1}|m)$.

We only need to further control $\|Hess_u\|_{L^2(B(p,r))}$ following
Cheeger and Colding's well-known argument based on their construction of a 
controlled cut-off function $\varphi$ supported on $B(p,2r)$, such that
$\varphi\equiv 1$ within $B(p,r+4\varepsilon r)$ and $r|\nabla
\varphi|+r^2|\Delta \varphi| \le 2C_{ctf}(m)$. Testing the Weitzenb\"ock
formula applied to $u$ against $\varphi$ on $B(p,2r)$, and applying integration
by parts and H\"older's inequality we see
\begin{align*} 
\fint_{B(p,2r)} \varphi |Hess_u|^2\ \le\ &\frac{1}{2}\int_{B(p,2r)} 
(|\Delta \varphi|+2\lambda (m-1))|\nabla u|^2
+\int_{B(p,2r)}(\Delta u)^2\varphi  +|\Delta u||\langle \nabla
\varphi,\nabla u\rangle|\\
\le\ &\left(8C_{ctf}(m)r^{-2}+ (m-1)(\varepsilon r)^2\right)K^2r^{-2}
+\frac{3}{2}\|\Delta u\|_{\bar{L}^2(B(p,2r))}^2.
\end{align*}
Therefore, by assuming $C_{ctf}(m)>1$ without loss of any generality, we get
\begin{align}
\fint_{B(p,r)}|Hess_{u}|\ \le\ 4mC_{ctf}(m)Kr^{-2} +2\|\Delta
u\|_{\bar{L}^2(B(p,2r))}.
\end{align}

Applying this estimate to (\ref{eqn: integral_derivative_lb1}) and integrating
in $t$ we immediately see
\begin{align*}
&\int_{\Phi^{-1}(\mathcal{R})}
\left(\left|\nabla^Tu\right|^2|J_k(\nabla \Phi)|\right)(\gamma_x(t))\
\dvol_g(x)-\int_{\Phi^{-1}(\mathcal{R})} \left|\nabla^Tu\right|^2(x)|J_k(\nabla
\Phi)|(x)\ \dvol_g(x)\\
 \ge\quad
&-16tmk^2C_{ctf}(m)(1+\Psi)^k |B(p,2r)| K^3r^{-4}-4t(1+\Psi)^k\|\Delta
u\|_{\bar{L}^2(B(p,2r)}|B(p,2r)|K^2r^{-2},
\end{align*}
and applying this inequality to (\ref{eqn: integral_variation_u}) and
integrating in $t$ we have
\begin{align*}
&\int_{\Phi^{-1}(\mathcal{R})}\left(u|J_k(\nabla \Phi)|\right)(\gamma_x(t))\
\dvol_g(x)-\int_{\Phi^{-1}(\mathcal{R})}u(x)|J_k(\nabla \Phi)|(x)\ \dvol_g(x)\\
\ge\quad &t\int_{\Phi^{-1}(\mathcal{R})}\left|\nabla^Tu\right|^2(x)
|J_k(\nabla \Phi)|(x)\ \dvol_g(x)-tk^2(1+\Psi)^{k-1}|B(p,2r)|\Psi K^2r^{-2}\\
&-8t^2mk^2C_{ctf}(m)(1+\Psi)^k|B(p,r)| K^3 r^{-4}-2t^2(1+\Psi)^k\|\Delta
u\|_{\bar{L}^2(B(p,2r)}|B(p,2r)|K^2r^{-2}.
\end{align*} 
These estimates should be compared with (\ref{eqn: integral_derivative_lb3}) and
(\ref{eqn: change_u_lb}). Now we bring in (\ref{eqn: change_integral_ub}), the
consequence of the smallness of the $\Phi$-fibers, to see that for any $t>0$,
\begin{align*}
&\int_{\Phi^{-1}(\mathcal{R})}\left|\nabla^Tu\right|^2|J_k(\nabla \Phi)|\
\dvol_g\\ 
\le\quad &4kt^{-1}(1+\Psi)^k|B(p,2r)|K\varepsilon
+k^2(1+\Psi)^{k-1}|B(p,2r)|\Psi K^2r^{-2}\\
&+8tmk^2C_{ctf}(m)(1+\Psi)^k|B(p,2r)|
K^3 r^{-4}+2t(1+\Psi)^k\|\Delta u\|_{\bar{L}^2(B(p,2r)}|B(p,2r)|K^2r^{-2}\\
\le\quad &16mk^2C_{ctf}(m)\left(t^{-1}\varepsilon+\Psi Kr^{-2}+tK^2r^{-4}
+t\|\Delta u\|_{\bar{L}^2(B(p,2r)}Kr^{-2}\right)(1+\Psi)^{k-1}|B(p,2r)|K.
\end{align*}
Now setting $t=\sqrt{\varepsilon}K^{-1}r^2$ we have
\begin{align*}
&\int_{\Phi^{-1}(\mathcal{R})}\left|\nabla^Tu\right|^2|J_k(\nabla \Phi)|\
\dvol_g\\  \le\quad  &16mk^2C_{ctf}(m)(2\sqrt{\varepsilon}+\Psi)
(1+\Psi_1)^{k-1}|B(p,2r)|K^2r^{-2}\\
&+16mk^2C_{ctf}(m)\sqrt{\varepsilon}\|\Delta u\|_{\bar{L}^2(B(p,2r)}
(1+\Psi)^{k-1}|B(p,2r)|K,
\end{align*}
and dividing by $|B(p,r)|$ on both sides we get, by (\ref{eqn: full_measure})
and the uniform boundedness of $|\nabla u|$ and $|J_k(\nabla \Phi)|$, that
\begin{align*}
r^2\fint_{B(p,r)}\left|\nabla^T u\right|^2|J_k(\nabla \Phi)|\ \dvol_g\ \le\
C_2(m,k) \left(\varepsilon^{\frac{1}{2}}+\Psi(\varepsilon, l^{-1}|m)\right)
\left(K^2 +\|\Delta u\|_{\bar{L}^2(B(p,2r))}Kr^2\right),
\end{align*}
where
\begin{align*}
C_2(m,k)\ :=\ \left(48mk^2C_{ctf}(m)2^k\right) \sup_{\varepsilon \in
(0,\varepsilon(m))}
\frac{\Lambda_{m,-\varepsilon^2}(2)}{\Lambda_{m,-\varepsilon^2}(1)}.
\end{align*}
is clearly independent of $\varepsilon \in (0,\varepsilon(m))$ and $r\in (0,1)$.
\end{proof}

With Theorem~\ref{thm: Ricci_tangential_L2}, we could now prove our main theorem
by invoking the Cheng-Yau gradient estimate in \cite{ChengYau75}:
\begin{proof}[Proof of Theorem~\ref{thm: main}] Since $(M,g)$ is Ricci flat, the
metric $g$ is analytic. Then the conditions of Theorem~\ref{thm:
Ricci_tangential_L2} are fulfilled. Now if $u\in C^{\infty}(B(p,2r))$ further
satisfies the eigenfunction equation $\Delta u=\theta u$, then by \cite[Theorem
6]{ChengYau75}, we could estimate, for some positive constant
$C_{CY}(m,\theta)$, that
\begin{align*}
\sup_{B(p,r)}r|\nabla u|\ \le\ C_{CY}(m,\theta)\sup_{B(p,2r)}|u|.
\end{align*}
Plugging this estimate in (\ref{eqn: u_C1_bound}) we get
$K=(1+C_{CY}(m,\theta))\|u\|_{L^{\infty}(B(p,2r))}$, whence the desired
estimate (\ref{eqn: main}) with $C(m,k,\theta)=C_2(m,k)(1+C_{CY}(m,\theta))$.
\end{proof}

\begin{remark}
In fact, for a general function $u\in C^{\infty}(B(p,2r))$, the quantity 
$\|u\|_{L^{\infty}(B(p,2r))}$ on the right-hand side of the estimate (\ref{eqn:
main}) could be replaced by certain $L^q$-average like 
\begin{align*}
\|u\|_{\bar{L}^q(B(p,2r))}+\|\Delta u\|_{\bar{L}^q(B(p,2r))},
\end{align*}
for some $q>2m$. Once we recall, through the work of Anderson \cite{Anderson}
and more generally the work by Saloff-Coste \cite{SC92}, that the Sobolev
inequality on a Riemannian manifold with Ricci curvature lower bound involves
the correct power of the volume $|B(p,2r)|$, this bound could be obtained by
following the routine of De Georgi iteration, starting from the Weitzenb\"ock
formula applied to $u$. Here we will save the extra lines of details, as our
main concern is about the ``nice'' functions like the eigenfunctions of the
Laplace operator.
\end{remark}

\subsection*{Acknowledgement} 
We would like to thank our advisor Xiuxiong Chen for his constant support. We
also thank Yu Li, Song Sun, Bing Wang, Ruobing Zhang and Yongzhe Zhang for their
interests in this work.


\begin{thebibliography}{99}
\bibliographystyle{plainnat}

\bibitem{AGS14a} Luigi Ambrosio, Nicola Gigli and Giuseppe Savar\'e, Calculus
and heat flow in metric measure spaces and applications to spaces with Ricci
bounds from below. \emph{Invent. Math.} 195 (2014), no. 2, 289-391.

\bibitem{AGS14b} Luigi Ambrosio, Nicola Gigli and Giuseppe Savar\'e, Metric
measure spaces with Riemannian Ricci curvature bounded from below, \emph{Duke
Math. J.} 163 (2014), no. 7, 1405-1490.

\bibitem{AHPT18} Luigi Ambrosio, Shouhei Honda, Jacobus W. Portegies, and David
Tewodrose, Embedding of $RCD^{\ast}(K,N)$ spaces in $L^2$ via eigenfunctions.
\emph{Preprint}, arXiv:1812.03712.

\bibitem{Anderson} Michael Anderson, The $L^2$ structure of moduli spaces of
Einstein metrics on $4$-manifolds. \emph{Geom. Funct. Anal.} 2 (1992), no. 1,
29-89.

\bibitem{Cheeger99} Jeff Cheeger, Differentiability of Lipschitz functions on
metric measure spaces. \emph{Geom. Funct. Anal.} 9 (1999), no. 3, 428-517. 

 \bibitem{Cheeger10} Jeff Cheeger, Structure theory and convergence in 
 Riemannian geometry. \emph{Milan J. Math.} 78 (2010), no. 1, 221-264. 


\bibitem{ChCoII} Jeff Cheeger and Tobias Colding, On the structure of spaces
with Ricci curvature bounded below. II. \emph{J. Differential Geom.} 54 (2000),
no. 1, 13-35.

\bibitem{ChCoIII} Jeff Cheeger and Tobias Colding, On the structure of spaces
with Ricci curvature bounded below. III. \emph{J. Differential Geom.} 54 (2000),
no. 1, 37-74.

\bibitem{ChNa14} Jeff Cheeger and Aaron Naber, Regularity of Einstein manfiolds
and the codimension $4$ conjecture. \emph{Ann. of Math.} 182 (2015), no. 3,
1093-1165.

\bibitem{ChengYau75} Shiu-Yuen Cheng and Shing-Tung Yau, Differential equations
on Riemannian manifolds and their geometric applications. \emph{Comm. Pure
Appl. Math.} 28 (1975), no. 3, 333 - 354.

\bibitem{ColdingI} Tobias Colding, Shanpe of manifolds with positive Ricci
curvature. \emph{Invent. Math.} 124 (1996), no. 1-3, 175-191.

\bibitem{ColdingII} Tobias Colding, Large manifolds with positive Ricci
curvature. \emph{Invent. Math.} 124 (1996), no. 1-3, 193-214.

\bibitem{ColdingIII} Tobias Colding, Ricci curvature and volume convergence.
\emph{Ann. of Math.} 145 (1997), no. 3, 477-501.

\bibitem{CoNa11} Tobias Colding and Aaron Naber, Sharp H\"older continuity of
tangent cones for spaces with a lower Ricci curvature bound and applications.
\emph{Ann. of Math.} 176 (2012), no. 2, 1173-1229. 





\bibitem{Ding02} Yu Ding, Heat kernels and Green's functions on limit spaces.
\emph{Comm. Anal. Geom.} 10 (2002), no.3, 475-514.

 \bibitem{doCarmo} Manfredo Perdig\~ao do Carmo, Riemannian geometry. Translated
 from the second Portuguese edition by Francis Flaherty. Mathematics: Theory \&
 Applications. \emph{Birkh\"auser Boston, Inc., Boston, MA}, 1992. 
 
 
  \bibitem{Fukaya87ld} Kenji Fukaya, Collapsing Riemannian manifolds to ones of
  lower dimensions. \emph{J. Differential Geom.} 25 (1987), no. 1, 139-156.

\bibitem{Fukaya87b} Kenji Fukaya, Collapsing of Riemannian manifolds and
eigenvalues of Laplace operator. \emph{Invent. Math.} 87 (1987), no. 3, 517-547.

\bibitem{Fukaya89} Kenji Fukaya, Collapsing Riemannian manifolds to ones with
lower dimension II. \emph{J. Math. Soc. Japan} 41 (1989), no. 2, 333-356.

\bibitem{GW00} Mark Gross and Pelham M. H. Wilson, Large complex structure
limits of $K3$ surfaces. \emph{J. Differential Geom.} 55 (2000), no. 3, 475 -
546.



\bibitem{HSVZ18} Hans-Joachim Hein, Song Sun, Jeff Viaclovsky and Ruobing Zhang,
Nilpotent structures and collapsing Ricci-flat metrics on $K3$ surfaces.
\emph{Preprint}, arXiv: 1807.09367.

\bibitem{Hirsch} Morris Hirsch, Differential topology. Graduate Texts in
Mathematics, 33. \emph{Springer-Verlag, New York-Heidelberg}, 1976.

\bibitem{HKRX18} Hongzhi Huang, Lingling Kong, Xiaochun Rong and Shicheng Xu,
Collapsed mnifolds with Ricci bounded covering geometry. \emph{Preprint},
arXiv: 1808.03774.

\bibitem{HW18} Shaosai Huang and Bing Wang, Rigidity of vector valued harmonic
maps of linear growth. \emph{Geom. Dedicata} 202 (2019), no. 1, 357-371.

\bibitem{LV09} John Lott and C\'edric Villani, Ricci curvature for
metric-measure spaces via optimal transport. \emph{Ann. of Math.} (2) 169
(2009), no. 3, 903-991.

\bibitem{NZ14} Aaron Naber and Ruobing Zhang, Topology and
$\varepsilon$-regularity theorems on collapsed manifolds with Ricci curvature
bounds. \emph{Geom. Topol.} 20 (2016), no. 5, 2575-2664.



\bibitem{SC92} Laurent Saloff-Coste, A note on Poincar\'e, Sobolev and Harnack
inqualities. \emph{Int. Math. Res. Not.} 1992, no. 2, 27-38.


\bibitem{Sturm06a} Karl-Theodor Sturm, On the geometry of metric measure spaces.
I. \emph{Acta Math.} 196 (2006), no. 1, 65-131.

\bibitem{Sturm06b} Karl-Theodor Sturm, On the geometry of metric measure spaces.
II. \emph{Acta Math.} 196 (2006), no. 1, 133-177.


\bibitem{Warner} Frank Warner, Foundations of differentiable manifolds and Lie
groups. Graduate Texts in Mathematics 94, \emph{Springer-Verlag, New
York-Berlin}, 1983.

\end{thebibliography}
\end{document}